\documentclass[11pt,twoside]{amsart}
\usepackage[T1]{fontenc}
\usepackage[utf8]{inputenc}
\usepackage[english]{babel}
\usepackage{graphicx}
\usepackage{amsmath}
\usepackage{amssymb}
\usepackage{enumerate}
\usepackage[all,cmtip]{xy}

\usepackage{amsthm}
\usepackage{amsfonts}
\usepackage{changepage}

\newtheorem{thm1}{Theorem}[section]
\newtheorem{theorem}[thm1]{Theorem}
\newtheorem{lemma}[thm1]{Lemma}
\newtheorem{corollary}[thm1]{Corollary}
\newtheorem{proposition}[thm1]{Proposition}

\newtheorem{thmx}{Theorem}

\theoremstyle{definition}
\newtheorem{definition}[thm1]{Definition}

\RequirePackage{ifthen}
\RequirePackage{calc}
\newcounter{exampleendflag}
\newcommand{\exendhere}{
  \setcounter{exampleendflag}{0} 
  \ifmmode
    \eqno
    \ensuremath{\blacktriangle}
  \else
    \hspace{\stretch{1}}
    \ensuremath{\blacktriangle}
  \fi
}
\newenvironment{example}{
  \setcounter{exampleendflag}{1}
  \begin{exx}
}{
  \ifthenelse{\value{exampleendflag}=1}{\exendhere}{} 
  \end{exx}
}

\theoremstyle{remark}
\newtheorem{remark}[thm1]{Remark}
\newtheorem{exx}[thm1]{Example}

\title{An intrinsic definition of the Rees algebra of a module}
\author{\gss}
\thanks{The author is supported by the Swedish Research Council, grant number 2011-5599.}
\newcommand{\saeden}{S\ae d\'en}
\newcommand{\stahl}{St\aa hl}
\newcommand{\gss}{Gustav \saeden\ \stahl}
\newcommand{\fF}{\mathcal{F}}

\newcommand{\fR}{\mathcal{R}}
\newcommand{\sym}{\operatorname{Sym}}

\newcommand{\moda}{\mathbf{{Mod}}_A}
\newcommand{\funa}{\mathbf{{Fun}}_A}
\newcommand{\grd}{{}^\vee}
\renewcommand{\bar}{\mathbf}
\newcommand{\im}{{\operatorname{im}}}  
\renewcommand{\hom}{\operatorname{Hom}}
\newcommand{\id}{\operatorname{id}}

\renewcommand{\phi}{\varphi}
\renewcommand{\epsilon}{\varepsilon}

\newcommand{\N}{\mathbb{N}}

\newcommand{\C}{\mathbb{C}}

\numberwithin{equation}{section}
\address{Department of Mathematics, KTH Royal Institute of Technology, SE-100 44 Stockholm, Sweden}
\email{gss@math.kth.se}
\begin{document}
\begin{abstract}
This paper concerns a generalization of the Rees algebra of ideals due to Eisenbud, Huneke and Ulrich that works for any finitely generated module over a noetherian ring. 
Their definition is in terms of maps to free modules. We give an intrinsic definition using divided powers.
\end{abstract}
\maketitle

\section*{Introduction}

Many definitions of Rees algebras of modules have been suggested, see e.g.\ \cite{MR1401523} and \cite{reesmodules}, as well as generalized Rees algebras of ideals as in \cite{MR2097164}. In this paper we study the definition of the Rees algebra of any finitely generated module over a noetherian ring due to Eisenbud, Huneke and Ulrich from \cite{reesbud}. Their definition is in terms of maps to free modules and one of our main results is that we find an intrinsic definition in terms of divided powers.
\begin{thmx}\label{intro:1}
Let $A$ be a noetherian ring and let $M$ be a finitely generated $A$-module. The Rees algebra $\fR(M)$ of $M$, as defined in \cite{reesbud}, is equal to the image of the canonical map
\[\sym(M)\to\Gamma(M^\ast)^\vee\]
of graded $A$-algebras. Here, $\sym(M)$ denotes the symmetric algebra of $M$ and 
\[\Gamma(M^\ast)^\vee=\bigoplus_{n\ge0}\hom_A\Bigl(\Gamma^n\big(\hom_A(M,A)\big),A\Bigr)\] denotes the graded dual of the algebra of divided powers $\Gamma(M^\ast)$ of the dual of the module $M$.
\end{thmx}

\par\vspace{\baselineskip}
\noindent\textbf{Background.} When resolving singularities in algebraic geometry, the concept of blow-ups of schemes is of great importance. 
 Given a scheme $X$ and a closed subscheme $Z$, the blow-up of $X$ along $Z$ is a new scheme $\operatorname{Bl}_Z(X)$ where $Z$ is replaced by the exceptional divisor. The blow-up is computed by taking the projective spectrum of the Rees algebra $R(\mathcal{I})=\bigoplus_{n\ge0} \mathcal{I}^n$ of the coherent sheaf of ideals $\mathcal{I}$ that corresponds to the closed subscheme~$Z$. This construction of the Rees algebra only works for ideals, but, as mentioned earlier, several generalizations have been considered. 
The first thing to note is that the Rees algebra of an ideal $I$ of a ring $A$ can be defined as the image of the map given by applying the symmetric algebra functor to the inclusion map $I\hookrightarrow A$, that is, as the image of the induced map $\sym(I)\to\sym(A)$. There are therefore many relations between the Rees algebra and the symmetric algebra. Since the symmetric algebra is defined for any $A$-module, this gives possibilities of natural generalizations. 

As the symmetric algebra functor does not preserve injections, the Rees algebra of an ideal $I$ is in general not equal to $\sym(I)$. Thus, to generalize the Rees algebra, some more sophisticated constructions need to be used, other than just considering the symmetric algebra. 
For instance, in \cite{MR1401523}, the authors give a definition for the Rees algebra of a finitely generated torsion-free module over a two-dimensional regular local ring with infinite residue field by embedding the module into a certain free module and taking the image of the map given by applying the symmetric algebra functor on this embedding. The authors of \cite{reesmodules} generalize this to a definition of the Rees algebra of a finitely generated module $M$ over a noetherian ring $A$, such that $M$ is free of constant rank when localizing at the associated primes of $A$, as the symmetric algebra of $M$ modulo its $A$-torsion. This does however not work in general as it may differ from the usual definition of the Rees algebra of an ideal when $A$ is not an integral domain. 
There are more examples of these kinds of generalizations, see e.g.\ \cite{MR1774764}, \cite{MR1806775}, \cite{MR1661272}, \cite{MR1709870}, \cite{MR1352554}, \cite{MR1308016}, \cite{MR878892} and \cite{MR1275840}. 

For completeness sake, even though we do not consider these here, we also mention that there are constructions that generalize the Rees algebra of an ideal in other directions, that is, not finding definitions that work for modules, but instead finds constructions over ideals with other properties than that of the Rees algebra. See for instance the quasi-symmetric algebra of an ideal defined in \cite{MR2097164}. 

Finally, in \cite{reesbud}, the authors give the generalization that we will study. It is a generalization of the Rees algebra which works for any finitely generated module over a noetherian ring, and it has the additional property of being functorial. To be more precise, they define the Rees algebra of a module $M$ as $\fR(M)=\sym(M)/\cap_gL_g$, where $g$ runs over all maps from $M$ to any free module $E$, and $L_g=\ker(\sym(g))$. Their definition is a true generalization of the Rees algebra of ideals and it also generalizes many of the constructions mentioned above. We will here study this definition with the purpose of finding another, equivalent, definition which is not expressed as the quotient of an infinite intersection. This we do  using  the theory of divided powers. 

\par\vspace{\baselineskip}
\noindent\textbf{Structure of the article.} We begin in Section~\ref{sec:vers} by reviewing the results of \cite{reesbud}. We study their definition of the Rees algebra of a finitely generated module over a noetherian ring and the notion of versal maps.



In Section~\ref{sec:bialg} we study general functors between the category of finitely generated $A$-modules and the category of $A$-algebras. We show that any right-exact functor $\fF$ gives a natural bialgebra structure on $\fF(M)$, for any module $M$. Furthermore, we show that the graded dual of a graded bialgebra has a canonical algebra structure.

In Section~\ref{sec:dp} we apply the general results of the previous section to the symmetric algebra and the algebra of divided powers, and show how these results give connections between the two algebra functors. 

Finally, in Section~\ref{sec:rees} we arrive at the result stated in Theorem~\ref{intro:1}, that shows that the Rees algebra of a finitely generated module $M$ over a noetherian ring can be defined as the image of the canonical map $\sym(M)\to\Gamma(M^\ast)^\vee$.

\par\vspace{\baselineskip}
\noindent\textbf{Acknowledgement.} I would like to thank D.\ Rydh for all the discussions that has made this article possible. Also, I am thankful to M.\ Boij for his help with the computer calculations and to R.\ Skjelnes for his helpful comments on this text. 

\section{Versal maps and the Rees algebra of a module}\label{sec:vers}\label{sec:1}
A ring will here always describe a commutative ring with a unit and an algebra over such a ring will always mean an associative commutative unital algebra. Throughout this paper, $A$ will denote a noetherian ring. We let $\sym(M)$ denote the symmetric algebra of a module $M$. In the paper \cite{reesbud}, the authors defined the Rees algebra of a finitely generated module over $A$ as follows.
\begin{definition}[\cite{reesbud}, Definition~0.1]\label{def:rees}
Let $M$ be a finitely generated $A$-module. We define the \emph{Rees algebra of $M$} as 
\[\fR(M)= \sym(M)/\cap_gL_g\]
where the intersection is taken over all homomorphisms $g\colon M\to E$ where $E$ runs over all free modules and $L_g=\ker\bigl(\sym(g)\colon\sym(M)\to\sym(E)\bigr)$.
\end{definition}
\begin{example}
If $M$ is free, then, as $g$ runs over all morphisms to free modules, the homomorphism $\id\colon M\to M$ is a homomorphism from $M$ to a free module and \[\sym(\id)\colon\sym(M)\to\sym(M)\] will be the identity, which is injective, so taking the intersection over all $g$ gives that \[\cap_gL_g=0.\]
Thus, if $M$ is free, then $\fR(M)=\sym(M)$.
\end{example}

One advantage of this definition is that it is functorial, that is, it gives a functor $\fR$ from the category of finitely generated $A$-modules to the category of $A$-algebras. Our goal will be to give another, equivalent, definition of this Rees algebra, which is more intrinsic to $M$ and not defined by mappings from $M$ to free modules. This we do in Section~\ref{sec:rees}. 

\begin{remark}\label{rem:psur}
Since the symmetric algebra is graded, the Rees algebra is also graded. Furthermore, the symmetric algebra preserves surjections and it follows that the Rees algebra does as well. 
\end{remark}

To compute the Rees algebra of a module $M$, the authors of \cite{reesbud} introduced the notion of a versal homomorphism.
\begin{definition}[\cite{reesbud}, Definition~1.2]
Let $M$ be a finitely generated $A$-module and let $F$ be a finitely generated and free $A$-module. A homomorphism $\phi\colon M\to F$ is \emph{versal} if any homomorphism $M\to E$, where $E$ is free, factors via $\phi$.
\end{definition}

\begin{lemma}\label{lem:versalgerrees}
If $\phi\colon M\to F$ is versal, then
\[\fR(M)=\fR(\phi)\]
where $\fR(\phi)=\im\bigl(\sym(\phi)\colon\sym(M)\to\sym(F)\bigr)$. 
\end{lemma}
\begin{proof}
If $\phi\colon M\to F$ is versal, then any $g\colon M\to E$, where $E$ is free, factors via $\phi$, that is, $g=h\circ\phi$ for some $h\colon F\to E$. By the functoriality of the symmetric algebra we get that $\sym(g)=\sym(h)\circ\sym(\phi)$.
Thus, $L_\phi=\ker(\sym(\phi))\subseteq\ker(\sym(g))=L_g$ for all $g$, so
\[\cap_gL_g=L_\phi=\ker\bigl(\sym(\phi)\bigr)\]
and therefore, 
\[\fR(M)=\sym(M)/\cap_gL_g=\sym(M)/\ker(\sym(\phi))=\im(\sym(\phi))=\fR(\phi).\qedhere\]
\end{proof}

\begin{remark}\label{ejversal}
Given an ideal $I\subseteq A$, the inclusion $I\hookrightarrow A$ is \emph{not} always versal. For instance, the inclusion $(x)\hookrightarrow \C[x]$ is not versal since the homomorphism $(x)\to \C[x]$ defined by $x\mapsto1$ does not factor through it. Indeed, there is no $\C[x]$-module homomorphism $\C[x]\to \C[x]$ sending $x$ to~$1$. Nevertheless, the Rees algebra $\fR(I)$ of an ideal $I$ can still be calculated by this inclusion map as the following result shows. 
\end{remark}
\begin{theorem}[\cite{reesbud}, Theorem 1.4]\label{thm:inclusion}
Let $A$ be a noetherian ring, let $I$ an ideal of $A$ and let $i\colon I\to A$ be the inclusion map. Then, $\fR(I)=\fR(i)=\im\bigl(\sym(i)\colon\sym(I)\to\sym(A)\bigr)$ is the classical Rees algebra of $I$.
\end{theorem}
Thus, Definition~\ref{def:rees} is a true generalization of the usual Rees algebra for ideals. To compute the Rees algebra of a module, we use versal maps and Lemma~\ref{lem:versalgerrees}. It is in fact easy  to construct a versal map from any finitely generated module $M$ over a noetherian ring. This follows from the following lemma and proposition.
\begin{lemma}\label{lem:faktoriserar}
Any map $M\to F$, where $F$ is free and finitely generated, factors canonically as $M\to M^{\ast\ast}\to F^{\ast\ast}\cong F$.
\end{lemma}
\begin{proof}
There are canonical maps from any module to its double dual, which gives the commutativity of the square:
\[\xymatrix{
M^{\ast\ast}\ar[r]&F^{\ast\ast}\\
M\ar[r]\ar[u]&F\ar[u]
}\]
Since $F$ is free and finitely generated it is canonically isomorphic to $F^{\ast\ast}$, so the commutative square reduces to the commutative diagram
\begin{equation*}
\begin{split}
\xymatrix{
M^{\ast\ast}\ar[dr]\\
M\ar[r]\ar[u]&F}
\end{split}
\end{equation*}
which gives the desired factorization.
\end{proof}
\begin{remark}
The previous result, and several other results throughout this paper, remain true when replacing free modules with projective modules. However, we restrict ourselves to free modules for clarifying the connections with the results of \cite{reesbud}. 
\end{remark}

\begin{proposition}[\cite{reesbud} Proposition 1.3]\label{prop:versur}
A homomorphism $\phi\colon M\to F$ is versal if and only if $\phi^\ast\colon F^\ast\to M^\ast$ is surjective.
\end{proposition}
\begin{proof}
Assume first that $M\to F$ is versal and consider any $g\in M^\ast$, that is, a homomorphism $g\colon M\to A$. Since $A$ is free and $M\to F$ is versal, it follows that there is some $h\colon F\to A$, that is, $h\in F^\ast$, lifting $g$. Hence $F^\ast\to M^\ast$ is surjective.

Conversely, assume that $\phi^\ast\colon F^\ast\to M^\ast$ is surjective. Let $E$ be a free $A$-module and let $g\colon M\to E$ be any homomorphism. We want to construct a map $h\colon F\to E$ such that $g=h\circ \phi$. Since $M$ is finitely generated we can assume that $E$ has finite rank, so $E^\ast$ is a free module of finite rank. In particular, $E^\ast$ is projective, which gives a map $h^\ast\colon E^\ast\to F^\ast$ such that the diagram
\[\xymatrix{
E^\ast\ar[r]^{g^\ast}\ar[dr]_{h^\ast}&M^\ast\\
&F^\ast\ar@{->>}[u]_{\phi^\ast}}\]
commutes.
Since $F$ and $E$ are free of finite rank it follows that $(F^\ast)^\ast=F$ and $(E^\ast)^\ast=E$, so dualizing again gives a map $h\colon F\to E$ such that $h\circ\phi^{\ast\ast}=g^{\ast\ast}$. We  have to show that $g=h\circ \phi$, and this follows from Lemma~\ref{lem:faktoriserar}. Indeed, both $g$ and $h$ factors via the canonical map $M\to M^{\ast\ast}$, which gives commutativity in the small triangles of the diagram
\[\xymatrix{&&M\ar[dl]\ar@/_1pc/[dll]_g\ar@/^1pc/[ddl]^\phi\\
E&M^{\ast\ast}\ar[l]_{g^{\ast\ast}}\ar[d]^{\phi^{\ast\ast}}\\
&F\ar[ul]^h}\]
and this implies that the whole diagram commutes.
\end{proof}

Since dualization is a contravariant left-exact functor, a versal map $M\to F$ has an injective double dual $M^{\ast\ast}\to F^{\ast\ast}$.
Therefore, given any finitely generated module $M$, there is a versal map $M\to F$ for some free and finitely generated module $F$. Indeed, since $A$ is noetherian and $M$ is finitely generated, it follows that $M^\ast$ is finitely generated, so we can find a surjection $F'\to M^\ast$ from some finitely generated and free module $F'$. Letting $F=(F')^\ast$ we get, by Proposition~\ref{prop:versur}, that the composition $M\to M^{\ast\ast}\hookrightarrow F$ is versal. 
\begin{remark}
Here we really need our standing assumption that $A$ is noetherian, since the dual of a finitely generated module over a non-noetherian ring need not be finitely generated and a non-finitely generated free module is not necessarily isomorphic to its double dual. This necessary assumption is not emphasized in some of the results of \cite{reesbud}.
\end{remark}

Furthermore, an immediate consequence of Lemma~\ref{lem:faktoriserar} and Proposition~\ref{prop:versur} is the following.

\begin{lemma}\label{lem:versalfakt}
If a homomorphism $\phi\colon M\to F$ is versal then it has a canonical factorization $M\to M^{\ast\ast}\hookrightarrow F$, where $M^{\ast\ast}\hookrightarrow F$ is injective.
\end{lemma}

\begin{remark}\label{rem:fnv}
Note that a homomorphism $\phi\colon M\to F$ that factors as $M\to M^{\ast\ast}\hookrightarrow F$ is not necessarily versal. Indeed, we have already noted in Remark~\ref{ejversal} that the inclusion $i\colon (x)\hookrightarrow \C[x]$ is not versal, and since $(x)$ is a free $\C[x]$-module it is canonically isomorphic to its double dual. Thus, it follows that we have such a factorization, 
\[\xymatrix{(x)^{\ast\ast}\ar@{^(->}[rd]\\(x)\ar[u]^{\cong}\ar@{^(->}[r]&\C[x]}\]
even though $i$ is not versal. In \cite{jag2}, we find a better characterization of versal maps in terms of coherent functors.
\end{remark}

The result of Lemma~\ref{lem:versalgerrees} tells us that we can calculate the Rees algebra of a module $M$ by using a versal map $M\to F$, for some finitely generated and free module $F$, and Lemma~\ref{lem:versalfakt} says that such a map factors as $M\to M^{\ast\ast}\hookrightarrow F$. Since the symmetric algebra does not preserve injections, this induces a commutative diagram
\[\xymatrix{\sym(M^{\ast\ast})\ar[dr]\\
\sym(M)\ar[u]\ar[r]&\sym(F)}\]
where the map $\sym(M^{\ast\ast})\to\sym(F)$ is not necessarily injective. Thus, $\fR(M)$ is not equal to the image of the map $\sym(M)\to\sym(M^{\ast\ast})$, so the factorization given by Lemma~\ref{lem:versalfakt} is not immediately applicable for computing the Rees algebra of $M$. However, in Section~\ref{sec:rees}, we  find another factorization of $\sym(M)\to\sym(F)$, using divided powers, such that the second map is injective. This factorization will also be induced by the result of Lemma~\ref{lem:versalfakt}, but in a less straightforward way. Finding this factorization will require results that we give in the following sections.

Let us, for now, instead discuss another aspect of the Rees algebra of a module, which shows an important difference between it and the symmetric algebra. It is well known that the symmetric algebra functor has the property that it is left adjoint to the forgetful functor from the category of $A$-algebras to the category of $A$-modules. Therefore, the symmetric algebra preserves colimits. 
The Rees algebra, however, does not preserve all colimits and is therefore not a left adjoint. This is a result that the following lemma and example will show.

\begin{lemma}\label{lem:sumver}
If $\phi\colon M\to F$ is versal then $\phi\oplus\phi\colon M\oplus M\to F\oplus F$ is versal. 
\end{lemma}
\begin{proof}
If $M\to F$ is versal, then $F^\ast\to M^\ast$ is surjective by Proposition~\ref{prop:versur}. Since dualization commutes with direct sums we have that $F^\ast\oplus F^\ast\to M^\ast\oplus M^\ast$ is surjective. Thus, $M\oplus M\to F\oplus F$ is versal by Proposition~\ref{prop:versur}.
\end{proof}

\begin{example}
We do \emph{not} in general have that $\fR(M\oplus N)=\fR(M)\otimes_A\fR(N)$, even though it holds for the symmetric algebra functor. Consider $A=\C[x]/x^2$ and $M=A/x$. Then the map $\phi\colon M\to A$ defined by $1\mapsto x$ is versal. Thus, by Lemma~\ref{lem:versalgerrees}, \[\fR(M)=\fR(\phi)=\im\bigl(\sym(\phi)\bigr).\]
We have that $\sym(M)=A[S]/xS$, where $S$ is an independent variable, that $\sym(A)=A[S]$, and that the map \[\sym(\phi)\colon A[S]/xS\to A[S]\] is defined by  $S\mapsto xS$. Thus, the image of $\sym(\phi)$ is $A[xS]=A[U]/(xU,U^2)$, so, by Lemma~\ref{lem:versalgerrees}, $\fR(M)=A[U]/(xU,U^2)$. By Lemma~\ref{lem:sumver} we have that $\phi\oplus\phi\colon M\oplus M\to A\oplus A$ is versal. Since $\sym(M\oplus M)=\sym(M)\otimes_A\sym(M)=A[S,T]/(xS,xT)$ we have that 
\[\sym(\phi\oplus\phi)\colon A[S,T]/(xS,xT)\to A[S,T]\]
is defined by $S\mapsto xS$ and $T\mapsto xT$. Thus, by Lemma~\ref{lem:versalgerrees},
\[\fR(M\oplus M)=A[xS,xT]=A[U,V]/(xU,U^2,xV,V^2,UV).\]
This is \emph{not} equal to 
\[\fR(M)\otimes_A\fR(M)=A[U]/(xU,U^2)\otimes_AA[V]/(xV,V^2)=A[U,V]/(xU,U^2,xV,V^2).\exendhere\]
\end{example}
Thus, we have seen that the Rees algebra does not preserve coproducts and is therefore not a left adjoint of any functor. 

\section{Colimit preserving functors and bialgebras}\label{sec:bialg}\label{sec:5}

We will here state some general results regarding functors from the category of $A$-modules to the category of $A$-algebras that we will need in the following sections. Also, we will define the notion of the graded dual of a graded algebra and show some results regarding these. 
\begin{proposition}\label{prop:bialg}
Let $\fF$ be a functor from the category of $A$-modules to the category of (graded) $A$-algebras. Suppose that $\fF$ preserves colimits and that the multiplication on $\fF(M)$, for any $A$-module $M$, is induced by the addition on $M$. Then, $\fF(M)$ has a natural structure of a commutative (graded) bialgebra.
\end{proposition}
\begin{proof}
Since the initial object is the empty coproduct it follows that $\fF(0)=A$. Let $M$ be an $A$-module. Applying $\fF$ to the diagonal $\Delta\colon M\to M\oplus M$ gives a comultiplication
\[\fF(M)\to\fF(M\oplus M)\cong\fF(M)\otimes_A\fF(M),\]
where the isomorphism comes from the fact that $\fF$ preserves colimits and, in particular, coproducts. The coassociativity and cocommutativity of this comultiplication follows from the symmetry of the diagonal morphism. As the multiplication on $\fF(M)$ is induced by the addition on $M$, it follows, by the commutativity of the diagram
\[\xymatrix{M\oplus M\ar[r]^-{\Delta\oplus\Delta}\ar[d]^+&(M\oplus M)\oplus (M\oplus M)\ar[d]^{+\oplus+}\\
M\ar[r]^-{\Delta}&M\oplus M}\]
that the comultiplication $\fF(\Delta)$ preserves multiplication. Since the comultiplication on $\fF(0)$ is the identity, $A=\fF(0)\to \fF(0)\otimes_A\fF(0)=A\otimes_AA=A$, we have that the comultiplication on $\fF(M)$ also preserves the unit and multiplication by elements of $A$. Thus, it follows that $\fF(\Delta)$ is an $A$-algebra homomorphism. Furthermore, applying $\fF$ to the zero map $M\to0$ gives a map $\epsilon\colon\fF(M)\to\fF(0)=A$. Let now $i_2\colon M\to 0\oplus M$ denote the inclusion into the second term. By the commutativity of 
\[\xymatrix{M\ar[d]_\Delta\ar[rd]^{i_2}\\M\oplus M\ar[r]_{(0,\id)}&0\oplus M}\]
we get a commutative diagram
\[\xymatrix{\fF(M)\ar[d]\ar[rd]^{\id}\\\fF(M)\otimes_A\fF(M)\ar[r]_-{\epsilon\otimes\id}&A\otimes_A\fF(M)}\]
showing that $\epsilon\colon\fF(M)\to A$ is the co-unit.
\end{proof}

\begin{remark}
A functor $\fF\colon\moda\to\funa$, with the same properties as in the statement of the proposition, even gives a natural Hopf algebra structure on $\fF(M)$, for any module~$M$. Indeed, by Proposition~\ref{prop:bialg}, $\fF(M)$ has a bialgebra structure. Letting $i\colon M\to M$ denote the module homomorphism defined by $x\mapsto-x$, we get a commutative diagram
\[\xymatrix{
M\ar[r]^-\Delta\ar[dr]&M\oplus M\ar[r]^{(\id,i)}&M\oplus M\ar[d]^+\\
&0\ar[r]&M
}\]
which, after applying $\fF$, shows that $\fF(i)$ is an antipodal map giving $\fF(M)$ a Hopf algebra structure.
\end{remark}

\begin{definition}
Let $B=\bigoplus_{n\ge0} B_n$ be a graded $A$-algebra. The \emph{graded dual} $B\grd$ of $B$ is then defined as
\[B^\vee=\bigoplus_{n\ge0}B_n^\ast,\]
where $B_n^\ast=\hom_A(B_n,A)$.
\end{definition}
\begin{remark}
Note the following.
\begin{enumerate}[(i)]
\item The notation for the graded dual that we use here coincides with the notation of the Matlis duality of graded modules over a local Noetherian graded ring, see, e.g., \cite[pp.141-143]{MR1251956}. However, these are very different notions and are only notationally similar. 
\item The graded dual $B^\vee$ need not be generated in degree $1$, even if $B$ is. See, e.g., Remark~\ref{rem:gammadualnotfinite}, where $\Gamma(M)$ is finitely generated in degree $1$ while its dual is not.
\item The graded dual is generally not an $A$-algebra, only an $A$-module. 
\end{enumerate}
\end{remark}

\begin{proposition}\label{prop:grdalg}
If $B=\bigoplus_{n\ge0}B_n$ is a graded $A$-bialgebra, then the $A$-module $B\grd$ has a natural structure of a graded $A$-algebra. 
\end{proposition}
\begin{proof}
If $B$ is a bialgebra, then we have homomorphisms $B\to B\otimes_AB$ and $B\to A$. Taking the graded dual of these two morphisms gives
\[(B\otimes_AB)\grd\to B\grd\]
and $A=A\grd\to B\grd$. Furthermore, there is a map $B\grd\times B\grd\to(B\otimes_AB)\grd$ sending an element $(f,g)\in B_m^\ast\times B_n^\ast$ to the homomorphism $\alpha_{f,g}\colon
B_m\otimes_AB_n\to A$, defined by $\alpha_{f,g}(b_m\otimes b_n)=f(b_m)\otimes g(b_n)$. This map is bilinear, hence induces a graded homomorphism $B\grd\otimes_AB\grd\to(B\otimes_AB)\grd$. We define the multiplication of $B\grd$ as the composition
\[B\grd\otimes_AB\grd\to(B\otimes_AB)\grd\to B\grd.\] 
That this multiplication is associative and commutative follows from the coassociativity and the cocommutativity of the comultiplication that it is induced by. Similarly, that $A\to B\grd$ is a unit follows by the fact that $B\to A$ is co-unit.
\end{proof}

\begin{remark}
The graded dual $B\grd$ is generally not a bialgebra. There is a homomorphism $B\grd\to(B\otimes_AB)\grd$ but since \[(B\otimes_AB)\grd\leftarrow B\grd\otimes_AB\grd\] goes in the wrong direction this map does not give any comultiplication structure on the $A$-algebra $B\grd$.
\end{remark}

\begin{remark}
In the sequel we write $f\otimes g$ for both the element in $B_m^\ast\otimes_AB_n^\ast$ and the image of $f\otimes g$ under the canonical map  $B_m^\ast\otimes_AB_n^\ast\to(B_m\otimes_AB_n)^\ast$, that is, $\alpha_{f,g}:=f\otimes g$.
\end{remark}

\begin{proposition}\label{prop:grdfalg}
Let $\fF$ be a functor from the category of $A$-modules to the category of graded $A$-algebras. Suppose that $\fF$ preserves colimits, and that the multiplication on $\fF(M)$, for any module $M$, is induced by the addition on $M$. Then $\fF(M)\grd$ is a graded $A$-algebra for any $M$. Furthermore, a module homomorphism $M\to N$ induces a homomorphism $\fF(N)\grd\to\fF(M)\grd$ of graded algebras.
\end{proposition}
\begin{proof}
By Proposition~\ref{prop:bialg} we have that $\fF(M)$ is a bialgebra. By Proposition~\ref{prop:grdalg} it thus follows that $\fF(M)\grd$ is a graded $A$-algebra. Let now $M\to N$ be an $A$-module homomorphism. There is a commutative diagram
\[\xymatrix{
M\ar[r]\ar[d]& N\ar[d]\\
M\oplus M\ar[r]&N\oplus N}\]
which, after applying $\fF$, gives a commutative diagram:
\[\xymatrix{
\fF(M)\ar[r]\ar[d]& \fF(N)\ar[d]\\
\fF(M)\otimes_A \fF(M)\ar[r]&\fF(N)\otimes_A\fF(N)}\]
Dualizing gives us the commutative diagram
\[\xymatrix{
\fF(M)\grd& \fF(N)\grd\ar[l]\\
(\fF(M)\otimes_A \fF(M))\grd\ar[u]&(\fF(N)\otimes_A\fF(N))\grd\ar[u]\ar[l]\\
\fF(M)\grd\otimes_A \fF(M)\grd\ar[u]&\fF(N)\grd\otimes_A\fF(N)\grd\ar[u]\ar[l]
}\]
which shows that the $A$-module homomorphism $\fF(N)\grd\to\fF(M)\grd$ preserves the multiplication. As $\fF(N)\grd\to\fF(M)\grd$ also maps the unit in $\fF(N)\grd$ to the unit in $\fF(M)\grd$, it follows that this is a homomorphism of graded $A$-algebras.
\end{proof}

Functors that preserve finite colimits are called right-exact. Every right-exact functor between abelian categories preserves surjections, and we conclude this section by showing that the same holds for right-exact functors $\moda \to \mathbf{Alg}_A$, even though $\mathbf{Alg}_A$ is not abelian.
\begin{lemma}\label{lem:coeqalg}
Let $B$ and $C$ be two $A$-algebras. The coequalizer of two homomorphisms \[f,g\colon B{
{\rightrightarrows}}C\] is the quotient $q\colon C\to C/I$, where $I$ is the ideal generated by $f(b)-g(b)$ for all $b\in B$.
\end{lemma}
\begin{proof}
It is clear that $q\circ f=q\circ g$. If $h\colon C\to D$ is a homomorphism to an $A$-algebra $D$ such that $h\circ f=h\circ g$, then it follows, by the universal property of the quotient, that there exist a unique map $C/I\to D$ factoring $h$. Hence $q\colon C\to C/I$ has the universal property of the coequalizer of $f$ and $g$.
\end{proof}

\begin{proposition}\label{prop:presurj}
A right-exact functor $\fF\colon \moda\to\mathbf{Alg}_A$ preserves surjections.
\end{proposition}
\begin{proof}
Let $f\colon M\twoheadrightarrow N$ be a surjection of $A$-modules. Then we can write $N$ as the coequalizer of the injection $i\colon\ker(f)\hookrightarrow M$ and the zero map $0\colon\ker(f)\to M$. Since $\fF$ preserves finite colimits it preserves coequalizers, so 
\[\fF(N)=\fF\bigl(\operatorname{coeq}(i,0)\bigr)=\operatorname{coeq}\bigl(\fF(i),\fF(0)\bigr).\]
By Lemma~\ref{lem:coeqalg} it follows that $\fF(N)=\fF(M)/I$, where $I$ is the ideal generated by the elements $\fF(i)(k)-\fF(0)(k)$ for all $k\in\fF(\ker(f))$. Thus, the map $\fF(M)\to\fF(M)/I=\fF(N)$ is surjective. 
\end{proof}

\section{The symmetric algebra and the algebra of divided powers}\label{sec:dp}\label{sec:6}

The symmetric algebra of a module is a very useful construction, see e.g.\ \cite[Appendix~A2.3]{eisen-comalg}. It gives a functor $\sym\colon \moda\to \mathbf{Alg}_A$, which is left adjoint to the forgetful functor $\mathbf{Alg}_A\to\moda$. Hence, the symmetric algebra functor preserves colimits. By Proposition~\ref{prop:presurj} it therefore also preserves surjections. Furthermore, since \[\sym(M)=\bigoplus_{n\ge0}\sym^n(M)\] is a graded algebra it follows by Proposition~\ref{prop:bialg} that $\sym(M)$ has a graded bialgebra structure for any module $M$. From Proposition~\ref{prop:grdfalg} it therefore follows that $\sym(M)^\vee$ is a graded $A$-algebra.

We will now consider another functor $\Gamma\colon\moda\to\mathbf{Alg}_A$. Many of the following definitions and results are taken from \cite{laksov} and can also be found in \cite{Roby1963} and \cite{Rydh}. 
\begin{definition}
Given an $A$-module $M$ we define the \emph{algebra of divided powers of $M$} as the $A$-algebra $\Gamma(M)=A[X(n,x)_{(n,x)\in\N\times M}]/I$, where $I$ is the ideal generated by 
\begin{enumerate}
\item $X(0,x)-1$,
\item $X(n,fx)-f^nX(n,x)$,
\item $X(m,x)X(n,x)-\binom{m+n}{m}X(m+n,x)$,
\item $X(n,x+y)-\sum_{i+j=n}X(i,x)X(j,y)$,
\end{enumerate}
for all $x,y\in M$, $m,n\in\N$ and $f\in A$. This algebra is graded, where the $A$-module $\Gamma^n(M)$ is generated, for every $n\in\N$, by the residue classes of all products $X(n_1,x_1)\cdot\ldots\cdot X(n_k,x_k)$, where $k\ge1$, for all $x_1,\ldots,x_k\in M$ and $n_1,\ldots,n_k\in\N$ with $n_1+\ldots+n_k=n$. We let $\gamma_M^n(x)$ denote the residue class of $X(n,x)$ in $\Gamma(M)$ and $\times$ denote the multiplication in the quotient.
\end{definition}

\begin{remark}
Similarly to the symmetric algebra, the degree one part of $\Gamma(M)$ is $M$. Also, the map $M\mapsto\Gamma(M)$ naturally extends to a functor $\Gamma\colon\moda\to\mathbf{Alg}_A$.
\end{remark}
There is another functor $\mathcal{E}\colon\mathbf{Alg}_A\to\moda$, that takes an $A$-algebra $B$ to the \emph{module of exponential sequences} $\mathcal{E}(B)$ of $B$. We will not define this module here and instead only refer the reader to \cite{Roby1963}. The existence of the functor $\mathcal{E}$ is however very important for the following theorem. 
\begin{theorem}[\cite{Roby1963}, Thm III.1]\label{thm:gpcol}
The algebra of divided powers $\Gamma$ is a left adjoint to the module of exponential sequences $\mathcal{E}$. Thus, the functor $\Gamma$ preserves colimits.
\end{theorem}
An immediate consequence of the previous theorem and Proposition~\ref{prop:presurj} is the following result.
\begin{corollary}\label{cor:gammaissurj}
 The algebra of divided powers preserves surjections.
\end{corollary}

For a finitely generated and free module, the algebra of divided powers has a particularly nice structure.
\begin{theorem}[\cite{Roby1963}, Thm IV.2]\label{thm:bas}
If $F$ is free with basis $x_1,...,x_n$, then the elements $\gamma_F^{m_1}(x_1)\times...\times\gamma_F^{m_n}(x_n)$, for all $m_1,...,m_n\in\N$, form a basis for the $A$-module $\Gamma(F)$. 
\end{theorem}

We also have the following property, which we will not use, but which we state in order to see the relation to Theorem~\ref{thm:symgrdg}. 
\begin{theorem}[\cite{laksov}, Theorem~5.6]\label{thm:visaan}
Let $M$ be an $A$-module. Then, there is a homomorphism of graded $A$-algebras
\[\Gamma(M^\ast)\to\sym(M)\grd.\]
If $M$ is free, then this map is an isomorphism.
\end{theorem}

Actually, we can see, similarly to the symmetric algebra, that $\Gamma(M)$ is not only a graded $A$-algebra but also a graded coalgebra. Indeed, since $\Gamma$ preserves colimits by Theorem~\ref{thm:gpcol}, the result follows from Proposition~\ref{prop:bialg}. We can also compute the co-multiplication \[\Delta\colon \Gamma(M)\to\Gamma(M)\otimes_A\Gamma(M)\] explicitly via the composition $\Gamma(M)\to\Gamma(M\oplus M)\cong\Gamma(M)\otimes_A\Gamma(M)$, as
\begin{align*}
\gamma_M^n(x)&\mapsto\gamma_M^n((x,x))=\gamma_M^n((x,0)+(0,x))=\\&=
\sum_{i+j=n}\gamma_M^i(x,0)\times\gamma_M^j(0,x)\mapsto
\sum_{i+j=n}\gamma_M^i(x)\otimes\gamma_M^j(x).
\end{align*}
\begin{example}If $F$ is free with basis $x_1,\ldots,x_n$, we have, by Theorem~\ref{thm:bas}, that the elements $\gamma_F^{m_1}(x_1)\times...\times\gamma_F^{m_n}(x_n)$, with $m_1,\ldots,m_n\in\N$, form a basis of $\Gamma(F)$ as a module. For future reference, we calculate the comultiplication $\Delta\colon\Gamma(F)\to\Gamma(F)\otimes_A\Gamma(F)$ applied  to a basis element of $\Gamma(F)$:
\begin{align*}
\gamma_F^{m_1}(x_1)\times\ldots&\times\gamma_F^{m_n}(x_n)\mapsto\gamma_F^{m_1}((x_1,x_1))\times\ldots\times\gamma_F^{m_n}((x_n,x_n))\mapsto\\
&\mapsto\left(\sum_{a_1+b_1=m_1} \gamma_F^{a_1}(x_1)\otimes\gamma_F^{b_1}(x_1)\right)\times\ldots\times\left(\sum_{a_n+b_n=m_n} \gamma_F^{a_n}(x_n)\otimes\gamma_F^{b_n}(x_n)\right)=\\
&=\sum_{\substack{a_j+b_j=m_j\\j\in\{1,\ldots,n\}}}\big(\gamma_F^{a_1}(x_1)\times\ldots\times\gamma_F^{a_n}(x_n)\big)\otimes\big(\gamma_F^{b_1}(x_1)\times\ldots\times\gamma_F^{b_n}(x_n)\big).
\end{align*}
\end{example}

Taking the graded dual of the algebra of divided powers $\Gamma(M)$ of a module $M$ gives, by Proposition~\ref{prop:grdfalg}, a multiplication \[\bullet\colon\Gamma(M)\grd\otimes_A\Gamma(M)\grd\to\Gamma(M)\grd,\] 
inducing an algebra structure on $\Gamma(M)^\vee$. This multiplication is graded and sends an element $u\otimes v\in\Gamma^i(M)^\ast\otimes_A\Gamma^j(M)^\ast$ to $u\bullet v\in\Gamma^{i+j}(M)^\ast$ defined, for all $\gamma\in\Gamma^{i+j}(M)$, by
\[(u\bullet v)(\gamma)=(u\otimes v)(\Delta(\gamma))\in A.\] 
In the case $\gamma=\gamma^{i+j}_M(x)$, for some $x\in M$, we  have that
\begin{multline*}\label{eq:mult1}(u\bullet v)\left(\gamma_M^{i+j}(x)\right)=(u\otimes v)\left(\Delta\left(\gamma_M^{i+j}(x)\right)\right)=(u\otimes v)\left(\sum_{a+b=i+j}\gamma_M^a(x)\otimes\gamma_M^b(x)\right)=\\ =\sum_{a+b=i+j}u\big(\gamma_M^a(x)\big)\otimes v\big(\gamma_M^b(x)\big)=u\bigl(\gamma_M^i(x)\bigr)\cdot v\bigl(\gamma_M^j(x)\bigr).\end{multline*}

\begin{theorem}\label{thm:gammadual}
If $F$ is free with basis $x_1,\ldots,x_n$, then there is a graded $A$-algebra isomorphism $A[\gamma_F^1(x_1)^\ast,\ldots,\gamma_F^1(x_n)^\ast]\to \Gamma(F)\grd$.
\end{theorem}
\begin{proof}
By Theorem~\ref{thm:bas} the elements $\gamma_F^{m_1}(x_1)\times\ldots\times\gamma_F^{m_n}(x_n)$, for all $m_1,\ldots,m_n\in\N$, form a basis of $\Gamma(F)$ as a module. Thus, the elements $\bigl(\gamma_F^{m_1}(x_1)\times\ldots\times\gamma_F^{m_n}(x_n)\bigr)^\ast$ constitute a basis of $\Gamma(F)\grd$. We will show that 
\[\bigl(\gamma_F^{m_1}(x_1)\times\ldots\times\gamma_F^{m_n}(x_n)\bigr)^\ast=\big(\gamma_F^1(x_1)^\ast\big)^{m_1}\bullet\ldots\bullet\big(\gamma_F^1(x_n)^\ast\big)^{m_n}.\]
It is trivially true in degree $1$. Let us introduce some multi-index notation and write \[\gamma_F^{\bar{m}}(\bar{x})=\gamma_F^{m_1}(x_1)\times\ldots\times\gamma_F^{m_n}(x_n)\] for $\bar{m}=(m_1,\ldots,m_n)\in\N^n$ and let us consider $\gamma_F^{\bar{m}}(\bar{x})^\ast\bullet\gamma_F^1(x_i)^\ast\in \Gamma^{m_1+\ldots+m_n+1}(F)^\ast$. In order to see what this product is, we apply it to the basis elements of degree $m_1+\ldots+m_n+1$ in $\Gamma(F)$, i.e., elements $\gamma_F^{\bar{r}}(\bar{x})$ where $\bar{r}=(r_1,\ldots,r_n)\in\N^n$ with $r_1+\ldots+r_n=m_1+\ldots+m_n+1$. Letting $\bar{e}_i$ denote the element $(0,\ldots,1,\ldots,0)\in\N^n$, with a $1$ in the $i$:th position, we have that 
\begin{align*}
\left(\gamma_F^\bar{m}(\bar{x})^\ast\bullet\gamma_F^1(x_i)^\ast\right)\left(\gamma_F^\bar{r}(\bar{x})\right)&= \left(\gamma_F^\bar{m}(\bar{x})^\ast\otimes\gamma_F^1(x_i)^\ast\right)\left(\Delta\big(\gamma_F^\bar{r}(\bar{x})\big)\right)=\\
&=\Big(\gamma_F^\bar{m}(\bar{x})^\ast\otimes\gamma_F^1(x_i)^\ast\Big) \left(\sum_{\substack{\bar{a}+\bar{b}=\bar{r}\\\bar{a},\bar{b}\in\N^n}}\gamma_F^\bar{a}(\bar{x})\otimes\gamma_F^\bar{b}(\bar{x})\right)=\\
&=\sum_{\substack{\bar{a}+\bar{b}=\bar{r}\\\bar{a},\bar{b}\in\N^n}}\left(\gamma_F^\bar{m}(\bar{x})^\ast\big(\gamma_F^\bar{a}(\bar{x})\big)\cdot\gamma_F^{1}(x_i)^\ast\big(\gamma_F^\bar{b}(\bar{x})\big)\right)=\\
&=\gamma_F^\bar{m}(\bar{x})^\ast\big(\gamma_F^{\bar{r}-\bar{e}_i}(\bar{x})\big)\cdot\gamma_F^1(x_i)^\ast\big(\gamma_F^1(x_i)\big)=\\
&=\begin{cases}
1,&\text{if }\bar{r}-\bar{e}_i=\bar{m},\\
0,&\text{else}.
\end{cases}
\end{align*}
Hence, \[\gamma_F^\bar{m}(\bar{x})^\ast\bullet\gamma_F^1(x_i)^\ast=\gamma_F^{\bar{m}+\bar{e}_i}(\bar{x})^\ast.\]
The result now follows by induction on the degree.
\end{proof}

\begin{theorem}\label{thm:symgrdg}
Let $M$ be a finitely generated $A$-module. Then, the canonical module homomorphism  $M^\ast\to\Gamma(M)^\vee$, sending $M^\ast$ into the degree $1$ part of $\Gamma(M)^\vee$, which is $M^\ast$, induces a natural homomorphism of graded $A$-algebras 
\[\sym(M^\ast)\to\Gamma(M)^\vee.\]
If $M$ is free, then this map is an isomorphism.
\end{theorem}
\begin{proof}
The homomorphism of graded $A$-algebras $\sym(M^\ast)\to\Gamma(M)^\vee$ is given by the universal property of the symmetric algebra. 

Suppose now that $M$ is free. Choosing a basis $x_1,\ldots,x_n$ of $M$ gives a dual basis $x_1^\ast,\ldots,x_n^\ast$ of $M^\ast$. Then, it follows that $\sym(M^\ast)\cong A[x_1^\ast,\ldots,x_n^\ast]$, and the result follows by Theorem~\ref{thm:gammadual}.
\end{proof}


\begin{remark}\label{rem:gammadualnotfinite}
The homomorphism $\sym(M^\ast)\to\Gamma(M)^\vee$ is not a surjection in general. Equivalently, the algebra $\Gamma(M)^\vee$ is generally not generated in degree $1$. 
For instance, let $A=\C[x]/x^2$ and let $M=A/x$. Then $M$ is the cokernel of the map $f\colon A\to A$, given by $f(1)=x$. Letting $0\colon A\to A$ denote the zero homomorphism, we have that $M$ is the coequalizer of $f$ and $0$. As $\Gamma$ preserves coequalizers by Theorem~\ref{thm:gpcol}, it follows that  $\Gamma(M)=\operatorname{coeq}(\Gamma(f),\Gamma(0))$. Therefore, by Lemma~\ref{lem:coeqalg}, $\Gamma(M)=\Gamma(A)/I$ where $I$ is the ideal generated by $\Gamma(f)(\gamma)-\Gamma(0)(\gamma)$ for all $\gamma\in\Gamma(A)$. In fact, it follows that the ideal $I$ is generated by $x\gamma_A^1(1)$. That is, we have an isomorphism $\Gamma(M)\cong A[y]/xy$, where $y=\gamma_A^1(1)$ has degree $1$. By taking the graded dual we get that $\Gamma(M)^\vee\cong A[xT,xT^2,xT^3,...]$ as a subring of $\Gamma(A)^\vee\cong\sym(A^\ast)\cong A[T]$.  
The polynomial algebra $A[xT,xT^2,xT^3,...]$ is not generated in degree $1$ and is not even finitely generated as an $A$-algebra. However, we will only be considering the image of the homomorphism $\sym(M^\ast)\to\Gamma(M)^\vee$, and  this is generated in degree $1$ since $\sym(M^\ast)$ is.
\end{remark}

\section{An intrinsic definition of Rees algebras}
\label{sec:rees}\label{sec:7}
The definition of the Rees algebra of a finitely generated module $M$, due to \cite{reesbud}, is in terms of maps from $M$ to free modules. We saw in Section~\ref{sec:vers} that the Rees algebra of $M$ can be computed using a versal map $M\to F$, where $F$ is a finitely generated and free module. There we also showed that any versal map factors via the double dual of $M$. 

Analogously, using the algebra of divided powers, we will now show that the Rees algebra of a module can be defined as the image of a canonical map out of $\sym(M)$, which is intrinsic to $M$ and not given by maps to free modules. 
\begin{lemma}\label{lem:injektivgdual}
Let $M$ be a finitely generated $A$-module and let $F$ be a finitely generated and free $A$-module. If $M\to F$ is a versal map, then there is an injective homomorphism $\Gamma(M^\ast)\grd\hookrightarrow\Gamma(F^\ast)\grd$ of graded $A$-algebras.
\end{lemma}
\begin{proof}
By Proposition~\ref{prop:versur}, we have that $F^\ast\twoheadrightarrow M^\ast$ is surjective. Using that $\Gamma$ preserves surjections by Corollary~\ref{cor:gammaissurj}, we get a graded surjection of algebras $\Gamma(F^\ast)\twoheadrightarrow\Gamma(M^\ast)$, that is, a surjection of modules $\Gamma^n(F^\ast)\twoheadrightarrow \Gamma^n(M^\ast)$ for all $n$. Dualizing these module homomorphisms gives injections $\Gamma^n(M^\ast)^\ast\hookrightarrow \Gamma^n(F^\ast)^\ast$ for all $n$. This shows that the graded $A$-algebra homomorphism $\Gamma(M^\ast)\grd\hookrightarrow\Gamma(F^\ast)\grd$, given by Proposition~\ref{prop:grdfalg}, is injective.
\end{proof}

\begin{theorem}\label{thm:main}
Let $A$ be a noetherian ring and let $M$ be a finitely generated $A$-module. Then,
\begin{equation*}
\fR(M)=\im\big(\sym(M)\to\Gamma(M^\ast)^\vee\big),
\end{equation*}
where the algebra homomorphism $\sym(M)\to\Gamma(M^\ast)^\vee$ is induced by the canonical module homomorphism from $M$ to the degree $1$ part of $\Gamma(M^\ast)^\vee$, which is $M^{\ast\ast}$.
\end{theorem}
\begin{proof}
Pick a versal map $M\to F$, for some finitely generated and free module $F$. 
By Lemma~\ref{lem:injektivgdual}, this gives an injective homomorphism $\Gamma(M^\ast)\grd\hookrightarrow\Gamma(F^\ast)\grd$ of graded $A$-algebras. Applying Theorem~\ref{thm:symgrdg} to the dual of $F$ gives a canonical isomorphism $\Gamma(F^\ast)\grd\cong\sym(F)$ of graded $A$-algebras. Furthermore, there is a canonical homomorphism $\sym(M)\to \Gamma(M^\ast)^\vee$ of graded $A$-algebras, given by the canonical map from $M$ to the degree~$1$ part of $\Gamma(M^\ast)^\vee$, which is $M^{\ast\ast}$. Putting all these together gives a map
\begin{equation}\label{eq:compositions}
\sym(M)\to\Gamma(M^\ast)^\vee\hookrightarrow\Gamma(F^\ast)^\vee\cong\sym(F).
\end{equation}
Restricting to the degree $1$ part of this composition of graded algebra homomorphisms gives a composition of module homomorphisms,
\begin{equation}\label{eq:comp2}
M\to M^{\ast\ast}\hookrightarrow F^{\ast\ast}\cong F,
\end{equation}
where the first map is the canonical map from $M$ to its double dual, the second map is the double dual of the versal map $M\to F$, and the isomorphism is the canonical isomorphism between $F$ and its double dual. By Lemma~\ref{lem:faktoriserar}, the composition \eqref{eq:comp2} equals the versal map $M\to F$. Thus, by adjunction, it follows that the map $\sym(M)\to\sym(F)$, given by applying the symmetric algebra functor on $M\to F$, is equal to the composition \eqref{eq:compositions}. Hence, we have a commutative diagram
\begin{equation*}
\begin{split}
\xymatrix{\Gamma(M^\ast)\grd\ar@{^(->}[dr]\\
\sym(M)\ar[u]\ar[r]&\sym(F)}
\end{split}
\end{equation*}
of  graded $A$-algebra homomorphisms, where $\Gamma(M^\ast)^\vee\hookrightarrow\sym(F)$ is injective. Combining this factorization with Lemma~\ref{lem:versalgerrees} gives the desired result.
\end{proof}

\begin{remark}
Note that all the previous steps goes through with the graded dual of the algebra of divided powers replaced with the graded dual of the symmetric algebra, \emph{except} for the algebra isomorphism $\Gamma(F^\ast)^\vee\cong\sym(F)$. There is a \emph{module} isomorphism $\sym(F^\ast)^\vee\cong\sym(F)$, but this is not an isomorphism of algebras since the multiplications are defined differently in these two $A$-algebras.
\end{remark}

By Theorem~\ref{thm:main}, we can give an equivalent definition of the Rees algebra of $M$ as the image of the canonical map $\sym(M)\to\Gamma(M^\ast)\grd$. Using this definition we can give new proofs of earlier results, such as the fact that the map $\fR\colon M\mapsto\fR(M)$ naturally extends to a functor $\moda\to\mathbf{Alg}_A$. Indeed, given a homomorphism $M\to N$, this gives a commutative diagram
\[\xymatrix@R=1.3em{
\sym(M)\ar@{->>}[dr]\ar[rr]\ar[dd]&&\sym(N)\ar@{->>}[dr]\ar[dd]\\
&\fR(M)\ar@{_(->}[dl]&&\fR(N)\ar@{_(->}[dl]\\
\Gamma(M^\ast)\grd\ar[rr]&&\Gamma(N^\ast)\grd}\]
which induces a map $\fR(M)\to\fR(N)$. The other defining properties of a functor follow by similar arguments.

\par\vspace{\baselineskip}
In \cite{jag2}, we show, analogously to Theorem~\ref{thm:main}, that the Rees algebra of a module $M$ can also be recovered from a canonical map of coherent functors.

\bibliography{references}{}
\bibliographystyle{amsalpha}

\end{document}